\documentclass[a4paper]{article}

\usepackage[utf8]{inputenc}
\usepackage[T1]{fontenc}

\usepackage{amssymb,amsmath,amsthm,amscd,amsfonts}
\usepackage{mathrsfs}

\usepackage{subfig}

\usepackage{hyperref}
\usepackage{url}
\usepackage{color}

\DeclareMathOperator{\diff}{Diff}
\def\Diff{\diff_+^2(\S)}

\def\C{\mathbb{C}}

\def\Z{\mathbb{Z}}
\def\P{\mathbb{P}^1}
\def\S{{\mathbb{S}^1}}

\def\D{\mathbb{D}^+}
\def\CD{\mathbb{D}^-}

\def\a{\alpha}

\def\g{\gamma}

\def\ssm{\smallsetminus}

\def\A{\mathcal{A}}
\def\U{\mathcal{U}}

\newcommand{\dd}{\mathrm{d}}

\DeclareMathOperator{\tr}{Tr}

\DeclareMathOperator{\id}{id}

\DeclareMathOperator{\I}{I}
\DeclareMathOperator{\GL}{GL}
\DeclareMathOperator{\SL}{SL}

\DeclareMathOperator{\MM}{M}
\DeclareMathOperator{\Vect}{Vect}

\theoremstyle{plain}
\newtheorem{thm}{Theorem}
\newtheorem{cor}{Corollary}[section]
\newtheorem{prop}{Proposition}[section]
\newtheorem{lemma}{Lemma}[section]

\newtheorem*{thm*}{Theorem}

\theoremstyle{definition}
\newtheorem{defn}{Definition}[section]

\usepackage{enumitem}

\title{Universal Schlesinger system and Birkhoff factorization}
\author{Laura \textsc{Desideri}%
\thanks{Laboratoire Paul Painlev\'e (U.M.R. CNRS 8524), U.F.R. de Math\'ematiques, Universit\'e Lille 1, 59655 Villeneuve d'Ascq Cedex, France -- Email: \texttt{Laura.Desideri@math.univ-lille1.fr}}}
\date{}

\begin{document}

\maketitle

\begin{abstract}
The aim of this paper is to establish an infinite dimensional generalization of the Schlesinger system --- a system of PDE's describing isomonodromic deformations of Fuchsian systems. This \emph{universal Schlesinger system} first appeared in a paper by Korotkin and Samtleben in the finite dimensional case (i.e. when it reduces to the finite dimensional Schlesinger system). We intend here to establish it in its full generality, and to give its geometrical meaning in terms of deformations of Birkhoff factorizations.

The Birkhoff factorization we are considering consists in finding a piecewise holomorphic matrix-valued function $Y$ on $\P\ssm\S$ which admits a prescribed multiplicative jump $G : \S\to \GL_N(\mathbb C)$ across the unit circle $\S$.

We explain how this factorization can be deformed  by the action of the group of diffeomorphisms of the circle, which acts by composition on the multiplicative jump $G$. The integrability equations of this  ``isomonodromic'' deformation is given by the universal Schlesinger system, which is then naturally expressed by means of Virasoro generators.
\end{abstract}

\bigskip



\section*{Introduction}
\addcontentsline{toc}{section}{Introduction}
The $p$-dimensional $N\times N$ \emph{Schlesinger system} is the following system of PDE's for $p$ unknown matrices $A_i(t)\in\MM_N(\C)$ depending on $p$ independent variables $t=(t_1, \ldots,t_p)\in\C^p$
\begin{equation}
  \dfrac{\partial A_i}{\partial t_j} = \dfrac{[A_i,A_j]}{t_i-t_j} , \; i\neq j, \qquad \quad
  \displaystyle\sum_{i=1}^p\dfrac{\partial A_i}{\partial t_j}=0.
\label{schlesinger}
\end{equation}
It describes isomonodromic deformations of Fuchsian systems. More precisely, consider a $N\times N$ Fuchsian system on the Riemann sphere $\P=\C\cup\{\infty\}$, i.e. a first order $N\times N$ linear ordinary differential system with only simple poles at $x=t_i$ ($\neq\infty$) 
\begin{equation}
\frac{\dd Y}{\dd x} = \A(x,t) Y, \qquad \text{with } \A(x,t) = \sum_{i=1}^p \frac{A_i(t)}{x-t_i} \  \text{ and } \sum_{i=1}^p A_i(t) = 0, 
\label{fuchsian-sys}
\end{equation}
whose residue matrices $A_i(t)\in\MM_N(\C)$ depend on the position of singularities $t=(t_1, \ldots,t_p)$. Then the residue matrices $A_i(t)$ solve the Schlesinger system if and only if the fundamental matrix solution $Y(x,t)$ of System~\eqref{fuchsian-sys} normalized at infinity by $Y(\infty,t)=\I_N$ has a $t$-independent monodromy representation. The aim of this paper is to generalize this deformation of Fuchsian systems by the Schlesinger system to a larger class of differential systems whose singular locus would not be finite any more, but would instead be a curve.

\paragraph{The system of Korotkin and Samtleben} In~\cite{KorotkinSamtleben}, D. Korotkin and H. Samtleben studied symmetries of the Schlesinger system by generalizing Oka\-mo\-to's equation to the $2\times2$ Schlesinger system. To this end, they rewrote for an arbitrary $N$ the Schlesinger system in a \emph{symmetric universal form}. Namely, they introduced the following differential operators (which satisfy the commutation relation of the Virasoro algebra):
\begin{equation}
\tilde L_m= \sum_{i=1}^p t_i^{m+1}\frac{\partial}{\partial t_i} \qquad\ \text{ for } m=-1, 0, 1,\ldots 
\label{vir-kor-samt}
\end{equation}
and the following change of dependent variables:
\[
B_n =  \sum_{i=1}^p t_i^{n+1} A_i \qquad\ \text{ for } n= 0, 1,\ldots 
\]
and proved that the Schlesinger system implies
\begin{equation}
\tilde L_m B_n =\sum_{k=1}^m \left[B_k,B_{m+n-k}\right] + nB_{m+n} \qquad \text{ for all } n\geq 0, m\geq -1.
\label{univ-kor-samt}
\end{equation}
Conversely, to derive the Schlesinger system from the infinite (dependent) set of equations~\eqref{univ-kor-samt}, it suffices to consider the equations for $n\leq p$ and $m\leq p$. What is remarkable in System~\eqref{univ-kor-samt} is its ``formal'' independence on the number $p$ of the poles $t_i$. The number and positions of the poles only enter in the definition of the differential operators $\tilde L_m$ and the variables $B_n$: System~\eqref{univ-kor-samt} provides us with a universal form of the Schlesinger system.

As Korotkin and Samtleben mentioned at the end of their work, the full set of equations~\eqref{univ-kor-samt} (including also the cases $n\leq -1, m\leq -2$) should have a geometrical meaning in its full generality, namely when the variables $B_m$ are independent. Such a system can not be defined from a Fuchsian system~\eqref{fuchsian-sys}. It should be instead a natural candidate for a generalization of isomonodromic deformations of systems of ODE's with an infinite number of singular points. We prove here that it is in fact the case: we will introduce the appropriate framework in which this universal Schlesinger system appears in its full generality, and give its geometrical interpretation as the integrability condition of a certain ``isomonodromic'' deformation of the Birkhoff factorization.

\paragraph{The Birkhoff factorization} consists in finding a holomorphic matrix-valued function having a prescribed multiplicative jump across a given curve. We consider the situation where all partial indices are zero. We will focus on the case where the given curve is the unit circle $\S$. Then, for a given (smooth) loop $G:\S\to \GL_N(\C)$ in the set on invertible $N\times N$ matrices, the Birkhoff factorization of data $\left( \S, G\right)$ is to find an invertible $N\times N$ matrix-valued function $Y(x)$, which is holomorphic on the inner and outer domains of the circle, and which admits pointwise limits $Y^+(\xi)$ and $Y^-(\xi)$ at any point $\xi\in\S$ satisfying
\[
Y^+(\xi) = Y^-(\xi) G(\xi).
\]
Precise hypothesis on $G$ will be set in Section 1 and some classical results will be given. When the jump matrix $G$ is piecewise constant, and the behavior of $Y$ at the singularities of $G$ appropriately prescribed, the Birkhoff factorization reduces to the Riemann--Hilbert problem for Fuchsian systems (see Section 1).

\bigskip

If one uses the terminology of the Schlesinger system, we can say that the loop $G$ plays here the role of monodromy, while the circle plays the role of varying singularity: the parameter of the deformation will be the diffeormorphisms of the circle, and we will say that  the loop $G$ is preserved if the group of diffeomorphisms acts on $G$ by  reparametrizations. We will prove that, for appropriate definitions of the unknown $B_n$ and of the differential operators $L_m$, the isomonodromic deformations are governed by  a universal Schlesinger~\eqref{univ-sch}, which is described ``for half'' by the system~\eqref{univ-kor-samt} of Korotkin and Samtleben.

\paragraph{The paper is organized as follows.} We start in \textbf{Section 1} by carefully stating the Birkhoff factorization in the form we need. If one sees the Birkhoff factorization as an inverse problem, the goal of the section is to state its corresponding direct problem, that is to say to introduce and study the class of first order matrix ordinary differential systems whose solutions solve the Birkhoff factorization. We prove that for a given $C^2$-smooth $G:\S\to\GL_N(\C)$, the solution $Y $ to the Birkhoff factorization of data $(\S,G)$, if it exists, solves the following system
\begin{equation}
\frac{\dd Y}{\dd x} = \A(x) Y, \qquad \text{with } \A(x) = \frac{1}{2i\pi} \oint_\S \frac{A(\xi)}{\xi-x}\dd \xi \ \text{ and } \oint_\S A(\xi)\dd \xi =0,
\label{A-intro}
\end{equation}
where the function $A:\S\to\MM_N(\C)$ is given by $A=Y^- \frac{\dd G}{\dd \xi} G^{-1} \left(Y^-\right)^{-1}$. The function $\A(x)$ is analytic on the exterior and the interior of the circle, and admits $A(\xi)$ as additive jump across the circle. The system~\eqref{A-intro} is the analogue of the Fuchsian system~\eqref{fuchsian-sys} with the circle as singular locus. The Fourier coefficients $B_n$, $n\in\Z$, of $A(\xi)$ are the coefficients of the series expansion of $\A(x)$ at $x=0$ and $x=\infty$. They will be the unknown functions of the universal Schlesinger system.

\vspace{.2cm}
In \textbf{Section 2} we define the ``isomonodromic'' deformation of system~\eqref{A-intro} and of the associated Birkhoff factorization by the action of the group of diffeomorphisms of the circle. We want to consider the singular points lying on the circle as varying parameters, and to describe those variations under which the monodromy $G$ is preserved. We consider the group $\Diff$ of the orientation preserving $C^2$-diffeo\-morphisms of the circle. It is a Banach manifold modeled on the vector space $\Vect ^2(\S)$ of $C^2$-vector fields along the circle.  We assume that system~\eqref{A-intro} depends on $\g\in\Diff$, that is to say $\A(x)=\A(x,\g)$. And so does its monodromy $G(\xi)=G(\xi,\g)$. We say that the monodromy is preserved if it is given by  a reparametrization by $\g$, i.e. $G(\xi,\g) = G_0\circ \g^{-1}(\xi)$.

Finally, we introduce the following complex vector fields $L_m$ ($m\in\Z$) on $\Diff$
\[
(L_m)_\g :\S\to\C, \qquad (L_m)_\g(\xi) = \g(\xi)^{m+1},
\]
which are obtained from Virasoro generators by a composition by $\g$, and span the tangent space of $\Diff$ at any $\g$. These vector fields enable us in the following section to write down the equations of the deformation.

\vspace{.2cm}
The aim of \textbf{Section 3} is to establish the integrability equations of the isomonodromic deformation, namely to prove that the coefficients $B_n(\g)$, which now depend on $\g$, solve the universal Schlesinger system
\begin{align*}
L_m\cdot B_n &= \ \sum_{k=0}^{n-1} \left[B_k,B_{m+n-k}\right] + nB_{m+n} \\
L_m\cdot B_{-n} &= -\sum_{k=-n}^{-1} \left[B_k,B_{m-n-k}\right] - nB_{m-n} 
\end{align*}
for all $m,n\in\Z$, with $ n\geq1$, if and only if we have an isomonodromic deformation of~\eqref{A-intro}. To this end, we first establish the action of Virasoro generators on the monodromy $G(\xi,\g)$ and characterize isomonodromic deformations by the equations
\[
\left(L_m\cdot G\right)_\g = - \xi^{m+1} \frac{\dd G}{\dd \xi} (\xi,\g).
\]
The strategy to establish the universal Schlesinger system is then similar to the one to establish the classical Schlesinger system, even if the objects and the tools are different. We exhibit an infinite dimensional Pfaffian system
\[
 \left\{
 \begin{aligned}
&\frac{\dd Y}{\dd x} = \A(x,\g) Y\\
&L_m\cdot Y = \Omega_m(x,\g) Y \qquad (m\in\Z)
 \end{aligned}
 \right.
 \]
which describes isomonodromic deformation, and whose integrability condition is equivalent to the universal Schlesinger system. The main tools are the Plemelj--Sokhotskii formula and the relations between series expansions and integral representations of the differential operators.

\bigskip

The approach followed in Section 1 could have a more algebraic formulation in the context of loop group factorizations (see~\cite{Guest}, \cite{PressleySegal}). In this point of view, one should mentioned the approach followed by M. Vajiac and K. Uhlenbeck in~\cite{VajiacUhlenbeck}, which is quite similar to the one of Sections 1 and 2 of the present paper: in the aim of establishing an action of Virasoro algebra on harmonic maps, they build infinitesimal transformations of a deformation by composition of a certain Birkhoff factorization.
\bigskip

As a conclusion, let us mentioned that the universal Schlesinger system presented here can be used in the case $N=2$ to describe deformations of minimal surfaces of disk type with a $C^3$ boundary curve, in the same spirit that the Schlesinger system provides deformations of minimal disks with a polygonal boundary curve (in the Garnier approach followed in~\cite{Desideri13}). During the deformation, the tangent direction of the boundary curve at any point is preserved, whereas the length  of the tangent vector is varying. This deformation might be used to solve the Plateau problem, if we succeed in prescribing this length function. In the same way that Garnier's solution to the Plateau problem reduces, when the minimal disk is planar, to Schwarz--Christoffel solutions to the Riemann mapping theorem, the minimal disks deformed by the universal Schlesinger system might reduce to the conformal maps studied by P. Wiegmann and A. Zabrodin in~\cite{WiegmannZabrodin},  and after by L. P. Teo~\cite{Teo,Teo2} and others. These conformal maps are deformed by the dispersionless Toda hierarchy, which should then be related to the universal Schlesinger system for $N=1$.


\section{Matrix-ODE formulation of the Birkhoff factorization}

\subsection{The Birkhoff factorization}
Let us state the Birkhoff factorization in the situation we are interested in.

\paragraph{The Birkhoff factorization}
Denote by $\D$ the unit disk in $\P=\C\cup\{\infty\}$, and by $\CD$ the outer domain of the unit circle ($\infty\in\CD$), so that $\P=\D\cup\S\cup\CD$.  

The Banach algebra $C^{k,\alpha}(\S)$ consists of those functions $f:\S\to\C$ with continuous derivative up to order $k$ whose $k$th-order derivative is Hölder continuous of exponent $\alpha$, $0<\alpha\leq1$, that is to say such that
\[
\Vert f \Vert_{k,\alpha} := \Vert f \Vert_\infty + \sup_{x\neq y} \frac{\vert f^{(k)}(y)-f^{(k)}(x)\vert}{\vert y-x\vert^\alpha} < +\infty.
\]

Then for a given Hölder continuous loop $G:\S\to \GL_N(\C)$, the Birkhoff factorization of data $\left( \S, G\right)$ is to find a $N\times N$ matrix-valued function $Y$ satisfying:
\begin{enumerate}
\item the functions $Y\big\vert _{\D}$ and $Y\big\vert _{\CD}$ are analytic  on $\D$ and $\CD$, and Hölder continuous on $\overline \D$ and on $\overline \CD$ respectively,
\item for all $\xi\in\S$ the two non tangential pointwise limits 
\[
Y^+(\xi)=\displaystyle\lim_{\substack{x\to\xi\\ x\in\D}} Y(x) \quad \text{ and } \quad Y^-(\xi)=\displaystyle\lim_{\substack{x\to\xi\\ x\in\CD}} Y(x)
\]
are related by
\begin{equation}
Y^+(\xi) = Y^-(\xi) G(\xi),
\label{jump}
\end{equation}
\item $Y(\infty)=\I_N$.
\end{enumerate}

\paragraph{Relation with the Riemann-Hilbert problem} Consider the case where the map $G$ is piecewise constant: take $p$ points $t_1,\ldots,t_p$ on the circle and assume that $G(\xi)= M_i\cdots M_1$ for $\xi\in\S$ between $t_i$ and $t_{i+1}$, where the constant matrices $M_i\in\GL_N(\C)$ satisfy $M_p\cdots M_1=\I_N$. The limits $Y^\pm(\xi)$ at $\xi\in\S\ssm\{t_i\}$ are taken as before in the pointwise sense, and we further assume that the behavior of $Y^\pm$ at the point $t_i$ is given by
\[
Y^\pm(\xi)(\xi-t_i)^\alpha \to 0 \quad \text{as } \xi\to t_i, \ \text{ for } 0\leq \alpha<1.
\]
It is then a classical result of Plemelj that, when at least one of the matrices $M_i$ is diagonalizable, the corresponding solution $Y$ to this Birkhoff factorization solves a Fuchsian system on the Riemann sphere of singularities $t_i$ and whose monodromy representation is generated by the matrices $M_i$. A detailed exposition of the relations between the Birkhoff factorization, the Riemann--Hilbert problem and Painlev\'e equations can be found in~\cite{FIKN}.

\paragraph{The Plemelj--Sokhotskii formula}
A major ingredient in the study of the Birkhoff factorization is the Plemelj--Sokhotskii formula (see~\cite{FIKN},~\cite{Giorgadze}). If one considers a continuous function $f:\S\to\C$, it defines a function
\begin{equation}
F(x) =  \frac{1}{2i\pi} \oint_\S \frac{f(\xi)}{\xi-x}\dd \xi
\label{F}
\end{equation}
which is analytic on $\D$ and $\CD$ and satisfies $F(\infty)=0$. The Cauchy operators $C_+$ and $C_-$ related to the unit circle $\S$
\[
C_\pm(f)(u) = \lim_{x\to u_\pm} \frac{1}{2i\pi} \oint_\S \frac{f(\xi)}{\xi-x}\dd \xi
\]
where $x\to u_\pm$ denotes the non-tangential limits from the $\pm$ side of $\S$ respectively, are then also well-defined at any point $u\in\S$. With our previous notations, $C_\pm(f)(\xi)=F^\pm(\xi)$. If we want the functions $C_\pm(f)$ to inherit the regularity of $f$, we need further assumptions on $f$. Indeed, the Cauchy operators $C_+$ and $C_-$ are bounded operators in $C^{0,\alpha}(\S)$. As operators in $C^{0,\alpha}(\S)$, they satisfy the \emph{Plemelj--Sokhotskii formula}
\[
C_\pm = \pm\frac12\mathbf{1} - \frac12 H
\]
where $\mathbf 1$ denotes the identical operator in $C^{0,\alpha}(\S)$, and $H$ is the Hilbert transform
\[
H(f)(u) = P.V. \frac{1}{i\pi} \oint_\S \frac{f(\xi)}{\xi-u}\dd \xi
\]
defined for $u\in\S$ as the Cauchy principal value of the improper integral. The Hilbert transform $H$ is also a bounded operator in $C^{0,\alpha}(\S)$.

In particular, we have $C_+-C_-=\mathbf 1$, that is to say for $f\in C^{0,\alpha}(\S)$
\[
F^+-F^-=f \qquad \text{ on }\S
\]
with $F$ given by~\eqref{F}. The solution~\eqref{F} to this boundary value problem becomes unique if one further requires that $F(\infty)=0$.

We have analogous results for functions in $L^p(\S)$ instead of Hölder continuous, see for instance~\cite{DeiftZhou02} or \cite{LitvinchukSpitkovskii}.

\paragraph{Existence and uniqueness}
The index $\text{Ind}_\S \det G$ is called the global index of the factorization. In the case where $\text{Ind}_\S \det G = 0$, the solution $Y$ to the Birkhoff factorization, if it exists, is uniquely determined by $G$.

Indeed, the function $y=\det Y: \P\ssm\S\to \C$ satisfies the one dimensional Birkhoff factorization $y^+ = y^- \det G$ on $\S$, and $y(\infty)=1$. When the global index is zero, we get, thanks to the Plemelj--Sokhotskii formula, that
\[
\det Y(x) = \exp\left( \frac{1}{2i\pi} \oint_\S \frac{\log \det G(\xi)}{\xi-x}\dd \xi\right).
\]
The functions $Y\big\vert _{\overline\D}$ and $Y\big\vert _{\overline\CD}$ are thus everywhere invertible. So if one considers two solutions $Y$ and $Z$ to the Birkhoff factorization, the function $Z Y^{-1}$ is holomorphic on the whole Riemann sphere and takes value $\I_N$ at $\infty$, i.e. $Y=Z$. In particular, the case $\text{Ind}_\S \det G = 0$ can be reduced to the $\SL_N(\C)$ case (see below).

The existence issue is a more complex question. Let us just state that the answer is positive on an open dense subset of the space of loops $G$. Let us introduce the following groups:
\begin{align*}
K &= C^{0,\a}\left(\S, \GL_N(\C)\right)\\
K^+ &= \left\{ Y:\D\to \GL_N(\C)  \text{ analytic } \big\vert  \ Y \text{ extends to } \S  \text{ and }  Y\in K \right\}\\
K^-_0 &= \left\{ Y:\CD\to \GL_N(\C)  \text{ analytic } \big\vert  \  Y \text{ extends to } \S  \text{ and }  Y\in K  \right.\\
& \hspace{7.5cm} \left. \text{ and } Y(\infty)=I \right\}.
\end{align*}
We then have:
\begin{thm}
The map $\mu: K^+\times K^-_0\to K$ defined for all $\xi\in\S$ by $\mu(Y,Z)(\xi)= \left(Z^-(\xi)\right)^{-1}\cdot Y^+(\xi)$ is a diffeomorphism from $K^+\times K^-_0$ onto an open dense subset $\tilde K$ of $K$.
\end{thm}

We say that $(K,K^+,K^-_0)$ is a local Manin triple. In particular, on $\tilde K$, the global index is zero (in fact, all the partial indices are zero).

In the following, we will require more regularity for the objects we consider. We will actually deal with the following groups:
\begin{align*}
L &= C^2\left(\S, \GL_N(\C)\right) \\
L^+ &= K^+ \cap L\\
L^-_0 &= K^-_0 \cap L
\end{align*}
and we have a diffeomorphism from $L^+\times L^-_0$ onto the open dense subset $\tilde L=\mu\left( L^+\times L^-_0 \right)$.


\subsection{Matrix-ODE formulation}

The Birkhoff factorization is an inverse problem. The aim of this section is to introduce the appropriate class of ordinary differential systems which enable us to state the ``direct problem''  corresponding to the Birkhoff factorization we have described in the previous section --- in the same way that the class of Fuchsian systems is associated with the Riemann--Hilbert problem. This amounts to consider the infinitesimal version of the previous theorem. 

\bigskip

Let $G$ be a loop in $\tilde L$. Then there exists a unique solution $Y$ to the Birkhoff factorization of data $(\S,G)$ and we have $Y\big\vert _{\D}\in L^+$ and $Y\big\vert _{\CD}\in L^-_0$. Let $Y_\infty$ denote this unique solution and consider its logarithmic derivative
\begin{equation}
\A(x) = \frac{\dd }{\dd x}\left(Y_\infty(x)\right)\cdot Y_\infty(x)^{-1}. 
\label{def-A(x)}
\end{equation}
From condition (i), we see that the matrix-valued function $\A(x)$ is analytic  on $\D$ and $\CD$, and since the functions $Y_\infty^+$ and $Y_\infty^-$ are $C^2$ on the circle $\S$, the functions $\A\big\vert _{\D}$ and $\A\big\vert _{\CD}$ extend to the circle $\S$ into $C^1$ (and thus $C^{0,\a}$) functions. The function $\A(x)$ has then an explicit additive jump across $\S$, given by $G$ and $Y_\infty$ (Lemma~\ref{lemma-A}), which, thanks to the the Plemelj--Sokhotskii formula, provides us with an explicit expression for the operator $\A(x)$(Proposition~\ref{prop-A}).

\begin{lemma}
For all $\xi\in\S$ the two pointwise limits $\A^\pm(\xi)$ are related by
\[
\A^+(\xi)= \A^-(\xi) + A(\xi)
\]
where $A:\S\to \MM_N(\C)$ is the $C^1$-function defined by 
\begin{equation}
A(\xi)=Y_\infty^-(\xi) \frac{\dd G(\xi)}{\dd \xi} G(\xi)^{-1} \left(Y_\infty^-(\xi)\right)^{-1}.
\label{A-jump}
\end{equation}
\label{lemma-A}
\end{lemma}

\begin{proof}
By definition and under the hypothesis
\[
\A^+(\xi) = \lim_{\substack{x\to\xi\\ x\in\D}} \left(\frac{\dd Y_\infty}{\dd x}\right) \cdot \left(Y_\infty^+(\xi)\right)^{-1}.
\]
Since ${Y_\infty}\big\vert _{\mathbb D^\pm}$ are analytic and $Y_\infty^\pm$ are differentiable, we can easily check that
\[
\lim_{\substack{x\to\xi\\ x\in\mathbb D^\pm}} \frac{\dd Y_\infty}{\dd x} = \frac{\dd \left(Y_\infty^\pm\right)}{\dd \xi} .
\]
But from~\eqref{jump}, we have $\frac{\dd Y_\infty^+}{\dd \xi} = \frac{\dd Y_\infty^-}{\dd \xi}G(\xi) + Y_\infty^-(\xi) \frac{\dd G}{\dd \xi} $, and thus
\begin{align*}
\A^+(\xi) &= \left(\frac{\dd Y_\infty^-}{\dd \xi}G(\xi) + Y_\infty^-(\xi) \frac{\dd G}{\dd \xi} \right) \cdot G(\xi)^{-1} \cdot Y_\infty^-(\xi)^{-1}\\
&= \A^-(\xi) + A(\xi).\qedhere
\end{align*}
\end{proof}

Thanks to the Plemelj--Sokhotskii formula, this enables us to obtain the following proposition.

\begin{prop}
If $G\in \tilde L$, then:
\begin{enumerate}
\item[(a)] there exists a $C^1$-function $A:\S\to \MM_N(\C)$ such that the function $\A(x)$ defined by~\eqref{def-A(x)} is given by
\begin{equation}
 \A(x) = \frac{1}{2i\pi} \oint_\S \frac{A(\xi)}{\xi-x}\dd \xi.
\label{A-Cauchy}
\end{equation}
In particular, $\A(x)$ is analytic  on $\D$ and $\CD$, and the functions $\A\big\vert _{\D}$ and $\A\big\vert _{\CD}$ extend onto $\overline \D$ and $\overline \CD$ respectively into  $C^1$-functions;
\item[(b)] moreover we have 
\[
\oint_\S A(\xi)\dd \xi =0 ,
\]
which is equivalent for the expansion of $\A(x)$ at infinity to satisfy
\begin{equation}
\A(x)=O\left(\frac{1}{x^2}\right) \qquad \text{as } x\to\infty.
\label{b}
\end{equation}
\end{enumerate}
\label{prop-A}
\end{prop}

\begin{proof}
Let us observe that the behavior at infinity~\eqref{b} is a direct consequence of the fact that $Y_\infty$ is analytic at infinity. In particular, $\A(\infty)=0$, which ends to establish condition (a). We can then see that the coefficient of $1/x$ in the expansion of $\A(x)$ at infinity is
\[
\frac{-1}{2i\pi} \oint_\S A(\xi)\dd \xi\, ,
\]
which thus vanishes.
\end{proof}

The function $A(\xi)$ can be considered as the analogue of the residue of $\A(x)$ along the singular curve $\S$. Since it is conjugated to the logarithmic derivative of $G$ at every point of $\S$ (but of course not globally conjugated), we get:

\begin{cor} The determinant of $G$ satisfies the following scalar equation
\[
\frac{\dd}{\dd\xi}\det G(\xi) = \tr A(\xi) \det G(\xi) .
\]
\label{cor-A-G}
\end{cor}


\paragraph{The direct problem} From now on, we will work with differential systems satisfying conditions (a) and (b) above. We start with a $C^1$-function $A:\S\to \MM_N(\C)$ satisfying $\oint_\S A(\xi)\dd \xi =0$, and we define the linear ordinary differential system
\begin{equation}
\frac{\dd Y}{\dd x} = \A(x) Y, \qquad \text{with } \A(x) = \frac{1}{2i\pi} \oint_\S \frac{A(\xi)}{\xi-x}\dd \xi.
\label{A}\tag{$A$}
\end{equation}
A fundamental solution of~\eqref{A} is a piecewise analytic function $Y:\P\ssm\S\to\GL_N(\C)$. Such a solution admits pointwise limits $Y^+(\xi)$ and $Y^-(\xi)$ at every point $\xi\in\S$, such that $Y^+,Y^-$ belong to the group $L$. Notice that two fundamental solutions $Y$ and $Z$ are related by two invertible matrices $C_+, C_-\in\GL_N(\C)$ such that
\[
Z(x) = 
\begin{cases}
Y(x) C_+ \qquad \text{ on } \D\\
Y(x) C_- \qquad \text{ on } \CD.
\end{cases}
\]

We define the \emph{monodromy} $G_Y:\S\to\GL_N(\C)$ of a fundamental solution $Y$
\[
G_Y(\xi) = \left( Y^-(\xi)\right)^{-1}Y^+(\xi),
\]
and $G_Y\in L$. For every fixed $\xi\in\S$, the matrix $G_Y(\xi)$ can be seen as the connection matrix between the continuation of the function $Y\big\vert _{\D}$ through $\S$ at $\xi$ and the function $Y\big\vert _{\CD}$. The (non analytic) continuation makes sense because the limit $Y^+(\xi)$  is an invertible matrix which can be considered as the initial condition of a unique fundamental solution $Z_\xi$ defined on $\CD$, and we have $Z_\xi (x)= Y(x)G_Y(\xi)$ for $x\in\CD$.

If one considers two fundamental solutions $Y, Z:\P\ssm\S\to\GL_N(\C)$ as above, then their monodromies satisfy $G_Z(\xi)=C_-^{-1} G_Y(\xi) C_+$, and they are not conjugated to each other. 


\paragraph{Reduction to $\SL_N(\C)$-systems} We can always reduce the system to the case $\tr(A(\xi))=0$, which is equivalent to the fact that a monodromy representative $G_0$ satisfies $\det G_0(\xi)=1$.

Indeed, starting with a fundamental solution $Y$ of the system~\eqref{A}, we define for any holomorphic function $\Phi:\P\ssm\S\to\C$ the function $Z=\Phi Y$, which is then a fundamental solution of the differential system:
\[
\frac{\dd Z}{\dd x} = \mathcal B(x) Z, \qquad \text{with } \mathcal B(x)= \A(x)+\frac{\Phi'}{\Phi}\I_N.
\]
By choosing for $\Phi$ the solution of the scalar equation
\[
\frac{\dd \Phi}{\dd x} =-\frac 1 N \tr\left(\A(x) \right)\Phi,
\]
satisfying $\Phi(\infty)=1$, we obtain a system such that $\tr\left(\mathcal B(x) \right)=0$. Since we also have
\[
\mathcal B(x) = \frac{1}{2i\pi} \oint_\S \frac{B(\xi)}{\xi-x}\dd \xi,
\]
the $C^1$-function $B:\S\to \MM_N(\C)$ satisfies $\tr(B(\xi))=0$. Corollary~\ref{cor-A-G} then shows that the monodromy of any fundamental solution has constant determinant. One can easily see that there exists a fundamental solution whose monodromy belongs to $\SL_N(\C)$. As a direct consequence, the determinant of any fundamental solution is then piecewise constant. 

Notice that we actually do not need any $C^1$ regularity here, continuity is sufficient.


\paragraph{Series expansion} Let us consider the series expansions of the operator $\A(x)$ at $x=0$ and at $x=\infty$. We define the coefficients $B_m$ ($m\in\Z$) of the expansions as follows:
\begin{equation}
\begin{aligned}
\A(x) &= \sum_{m=0}^{+\infty}\frac{B_m}{x^{m+1}} & \ \text{on } \CD,\\
\A(x) &= -\sum_{m=1}^{+\infty} B_{-m} x^{m-1}  & \ \text{on } \D.
\end{aligned}
\label{A-expansion}
\end{equation}
From condition (b) in Proposition~\ref{prop-A}, we have $B_0=0$.

\begin{lemma}
For all $m\in\Z$ the coefficients $B_m$ are given by
\[
B_m =  \frac{-1}{2i\pi} \oint_\S \xi^mA(\xi)\dd \xi .
\]
They are the Fourier coefficients of the function $A:\S\to \MM_N(\C)$, which expands as
\[
A(\xi) = \A^+(\xi)-\A^-(\xi) = - \sum_{m\in\Z} \frac{B_m}{\xi^{m+1}}.
\]
\label{lemma-B_m}
\end{lemma}

\begin{proof}
On $\CD$, we simply use the expansion
$
\frac{1}{\xi-x} = -\frac1x \sum_{m=0}^{+\infty} \left(\frac{\xi}{x}\right)^m
$
in the expression~\eqref{A-Cauchy} of $\A(x)$ and get
\[
\A(x) = -\sum_{m=0}^{+\infty} \frac{1}{x^{m+1}} \frac{1}{2i\pi} \oint_\S \xi^mA(\xi)\dd \xi.
\]
On $\D$, the argument is the same.
\end{proof}


\section{The action of Virasoro generators}

The unit circle $\S$ plays the role of singularity of the system~\eqref{A}, and the loop $G$ the one of monodromy. We want to consider the singular points lying on the circle as varying parameters, on which would depend the system~\eqref{A}, and in particular, its monodromy $G$. Our aim is to describe those variations under which the monodromy is preserved. Since we restrict ourselves to Birkhoff factorizations along the unit circle, the only possible variations of the singular locus $\S$ are described by diffeomorphisms of the circle. The strategy is then:
\begin{itemize}
\item assume the system~\eqref{A} depends on a parameter $\g$ laying on the group of diffeomorphisms of the circle ;
\item characterize those deformations of ~\eqref{A} by $\g$ for which the monodromy $G$, which now depends on $\g$, is preserved. We will say that the monodromy is \emph{preserved} when its dependence on $\g$ is a reparametrization by $\g$ ;
\item write the equations of such an isomonodromic deformation. This step requires a appropriate description of the tangent space of the group of diffeomorphisms of the circle.
\end{itemize}

\vspace{.2cm}
We consider the group $\Diff$ of the orientation preserving $C^2$-diffeo\-morphisms of the circle, equipped with the topology of uniform convergence of the mappings $\g:\S\to\S$ and their first and second derivatives. It is a Banach manifold modeled on the vector space $\Vect ^2(\S)$ of $C^2$-vector fields along the circle (see below). But $\Diff$ is not a Banach Lie group, since the structures of group and of Banach manifold are not compatible. The reason of our choice of $\Diff$ instead of the more usual group of \emph{smooth} diffeomorphisms of the circle, which is a Fréchet Lie group, is that we need to apply a Fröbenius type theorem to integrate the equations of deformation.

\bigskip
The group $\Diff$ acts on $L$ by composition (or reparametrization)
\[
 \Diff\times L \to L, \qquad  (\g,G) \mapsto G\circ\g^{-1}.
\]
Let us fix an initial loop $G_0\in L$, and denote by $G_\gamma:=G_0\circ\g^{-1}\in L$. This defines a function $G$
\[
G : \S\times\Diff \to \GL_N(\C),
\qquad  G(\xi,\gamma) = G_\gamma(\xi) = G_0\circ\gamma^{-1}(\xi).
\]
In particular, $G_0=G_e$ where $e$ is the identity $e=\id_\S$. Moreover, if a diffeomorphism $\g$ is close enough to the identity $e$, and if $G_0$ is an element of the subspace $\tilde L$, then $ G_0\circ\g^{-1} \in\tilde L$. The Birkhoff factorization then provides us with a unique $Y_\infty(x,\g)$ and thus a unique $\A(x,\g)$ of monodromy $G_\g$. Since, in general, the diffeomorphism $\g$ does not extend into a diffeomorphism of the Riemann sphere, the deformed operator $\A(x,\g)$  cannot be explicitly expressed in terms of the original one $\A_0(x)$. Let $\U$ be an open neighbourhood of the identity in $\Diff$, such that for any $\g\in\U$ we have $ G_\g \in\tilde L$ ($\U$ depends on $G_0$). The diffeomorphisms $\g\in\U$ are thus the parameter of the deformation described above. We will call such a deformation, obtained by reparametrizations of $G$, an \emph{isomonodromic deformation}.

\bigskip
To describe the dependence on $\g$, we need to describe the tangent space to $\Diff$. For  every $\g\in\Diff$, there is an isomorphism between the tangent space $T_\g\Diff$ of $\Diff$ at $\g$ and the vector space $\Vect ^2(\S)$, given by
\[
\Vect^2(\S) \to T_\g\Diff , \qquad V\mapsto V\circ\g \, .
\]
Virasoro generators, which can be intuitively written as
\[
 \xi^{m+1}\frac{\dd}{\dd\xi} \quad (m\in\Z)\, ,
\]
topologically generate the complexification of the space $\Vect^2(\S)$. To check this, one should consider Féjer sums instead of Fourier sums as it is usually the case when we consider smooth diffeomorphisms. Indeed, the Féjer sum of a continuous function on the circle is also finitely generated by the $\xi^m$, but uniformly converges to the function (and so on for the derivatives). Taking the images of Virasoro generators by the above isomorphism, we get the following complex vector fields $L_m$ on $\Diff$: 
\begin{equation}
 (L_m)_\g : \xi\mapsto \g(\xi)^{m+1} \in T^\C_\g\Diff 
\label{def-vir}
\end{equation}
($m\in\Z$). For every $\g\in\Diff$, the complex vector space generated by the vectors $(L_m)_\g$ is a dense sub-space of the complexified tangent space $T^\C_\g\Diff$ of $\Diff$ at $\g$. As vector fields on $\Diff$ , the $L_m$ satisfy the commutation relation of the Virasoro generators: $\left[L_m, L_n\right] = (n-m)L_{m+n}$.

For a function $\Phi$ on $\Diff$, we can thus define the Lie derivative of $\Phi$ along $L_m$ by $ \left(L_m\cdot \Phi\right)_\g=\frac{\dd}{\dd t}\Big\vert_{t=0} \Phi\circ\varphi_m(t,\g)$, where $\varphi_m(t,\g)$ is the flow of $L_m$ at $\g$. 

\bigskip
Let us point out the relation of the vectors fields $L_m$ with the operators $\tilde L_m$ given by~\eqref{vir-kor-samt} and introduced by Korotkin and Samtleben. Consider the singular set of the associated Fuchsian system: it is formed of the $p$-tuples of ordered points on the circle $\mathcal B_p = \{(t_1,\ldots,t_p)\in(\S)^p \ \vert \ \arg(t_1)<\cdots<\arg(t_p)<\arg(t_1)+2\pi \}$. Fix $p$ points on the circle, for instance $\tau_j=e^{2i\pi j/p}$ ($j=1,\ldots,p$), and define the map
\begin{align*}
T : \Diff & \to \mathcal B_p\\
 \g & \mapsto (\g(\tau_1),\ldots,\g(\tau_p))\, .
\end{align*}
For a function $f$ on $\Diff$ which admits a factorization $f=g\circ T$ by the map $T$ with $g$ a function on $\mathcal B_p$, we then have
\[
\left(L_m\cdot f\right)_\g = \sum_{j=1}^p \g(\tau_j)^{m+1}\frac{\partial g}{\partial t_j}  (\g(\tau_1),\ldots,\g(\tau_p))= \left(\tilde L_m \cdot g\right)_{T(\g)}.
\]


\section{Universal Schlesinger system}
\subsection{Definition of the deformation and main result}

We want to consider a family of linear differential systems~\eqref{A} depending on $\g\in\U$. Let us start with a function $A:\S\times\U\to\MM_N(\C)$ which is $C^1$ with respect to both $\xi\in\S$ and $\g\in\U$, and verifies $\oint_\S A(\xi,\g) \dd \xi =0$ for all $\g\in\U$. It defines the family of systems
\begin{equation}
\frac{\dd Y}{\dd x} = \A(x,\g) Y
\label{Ag}\tag{$A_\g$}
\end{equation}
where
\begin{equation}
 \A(x,\g) = \frac{1}{2i\pi} \oint_\S \frac{A(\xi,\g)}{\xi-x}\dd \xi =
 \begin{cases}
\displaystyle \quad \sum_{m\geq1}\dfrac{B_m(\g)}{x^{m+1}} &\text{if } \vert x\vert >1\\
\displaystyle -\sum_{m\geq1} B_{-m}(\g) x^{m-1}  &\text{if } \vert x\vert <1.
\end{cases}
\label{exp-Ag}
\end{equation}
The coefficients $B_m(\g)$ are $C^1$-functions on $\U$.

\begin{defn}
A fundamental solution $Y:\left(\P\ssm \S\right)\times\U\to\GL_N(\C)$ of $(A_\g)$ is said to be $M$-invariant if its monodromy loop $G:\S\times\U\to\GL_N(\C)$ satisfies
\begin{align*}
G(\xi,\g) &= G(\g^{-1}(\xi),e)\\
&=G_e\circ\g^{-1}(\xi).
\end{align*}

Moreover, the family of differential systems $\left(A_\g, \g\in\U\right)$ is said to be isomonodromic if it admits an $M$-invariant solution.
\end{defn}

The aim of the present section is to prove the following theorem.

\begin{thm}
The family of differential systems $\left(A_\g, \g\in\U\right)$ admits an $M$-invariant solution $Y(x,\g)$ satisfying $Y(\infty,\g)=\I_N$ if and only if the coefficients $B_n(\g) (n\in\Z)$ satisfy the following  non linear differential system:
\begin{equation}
\begin{split}
L_m\cdot B_n &= \ \sum_{k=0}^{n-1} \left[B_k,B_{m+n-k}\right] + nB_{m+n} \\
L_m\cdot B_{-n} &= -\sum_{k=-n}^{-1} \left[B_k,B_{m-n-k}\right] - nB_{m-n} 
\end{split}
\label{univ-sch}
\end{equation}
for all $m,n\in\Z$, with $ n\geq1$.
\label{thm-sch-univ}
\end{thm}

After Korotkin and Samtleben, we call the above system~\eqref{univ-sch} the \emph{universal Schlesinger system}. Indeed, it contains all the finite dimensional $N\times N$ Schlesinger systems.

Before going into the proof of Theorem~\ref{thm-sch-univ}, let us establish an equivalent form of the universal Schlesinger system.

\begin{lemma}
We have the following equivalent form for the system~\eqref{univ-sch}: for all $m,n\in\Z$, with $ m\geq0$
\begin{equation}
\begin{split}
L_m\cdot B_n &= \ \sum_{k=0}^{m} \left[B_k,B_{m+n-k}\right] + nB_{m+n} \\
L_{-m}\cdot B_n &= -\sum_{k=-m+1}^{-1} \left[B_k,B_{-m+n-k}\right] + nB_{-m+n} 
\end{split}\label{univ-sch-2}
\end{equation}
where the sum is assumed to be zero whenever the subscript is larger than the upperscript.
\label{lemma-sym}
\end{lemma}

\begin{proof}
Let us first notice that for all $p,q\in\Z$, we have 
\[
\sum_{k=p}^q \left[B_k,B_{p+q-k}\right] = \sum_{k=p+1}^{q-1}\left[B_k,B_{p+q-k}\right]= 0
\]
since every bracket appears twice in each sum with opposite signs (up to a discussion on the parity of $q-p$). Let us fix $ m\geq0$ and $ n\geq1$. Thanks to the previous remark, by distinguishing the cases $n\leq m$ and $n> m$, we get
\[
\sum_{k=0}^{n-1} \left[B_k,B_{m+n-k}\right] = \sum_{k=0}^{m} \left[B_k,B_{m+n-k}\right]
\]
and the equations number $(m,n)$ in~\eqref{univ-sch} and in~\eqref{univ-sch-2} are equivalent. As for the equations number $(-m,-n)$. Now for the equations number $(-m,n)$, it suffices to see that the identity
\[
 \sum_{k=-m+1}^{n-1}\left[B_k,B_{-m+n-k}\right] =0
\]
implies 
\[
-\sum_{k=-m+1}^{-1} \left[B_k,B_{-m+n-k}\right] = \sum_{k=0}^{n-1} \left[B_k,B_{-m+n-k}\right] \,.
\]
The same argument holds for the equations number $(m,-n)$.
\end{proof}


\subsection{Equation of the deformation}

We are now going into the proof of Theorem~\ref{thm-sch-univ}. We first give a characterization of isomonodromic deformations in terms of Virasoro generators.

\begin{prop}
Let $\Phi$ be a $C^1$-function defined on $\S\times\Diff$. Let $\Phi_e:\S\to\C$ be the function $\Phi(\cdot,e)$.

The function $\Phi$ is a composition, that is to say $\Phi(\xi,\g)=\Phi_e\circ\g^{-1}(\xi)$, if and only if for all vector field $w$ on $\Diff$ the Lie derivative of $\Phi$ along $w$ is given by
\[
w\cdot \Phi (\xi,\g) = -w_\g\circ\g^{-1}(\xi)\frac{\dd \Phi}{\dd\xi}(\xi,\g)
\]
for all $(\xi,\g)\in\S\times\Diff$.
\label{prop-composition}
\end{prop}

Notice that for all $\g\in\Diff$ we have $w_\g\circ\g^{-1}\in\Vect^2(\S)$.

\begin{proof}
The fact that the function $\Phi$ is a composition $\Phi_\g=\Phi_e\circ \g^{-1}$ is equivalent for the function $\Phi_\g\circ\g:\S\to\C$ to be independent of $\g$, that is to say for its Lie derivative with respect to any vector field $w$ on $\Diff$ to vanish (since $\Diff$ is connected).  But we have
\begin{align*}
w \cdot \left(\Phi_\g\circ\g \right) (\xi)  
&=\left[w\cdot\left(\Phi(\g (\xi),\g)\right)\right]_\g\\
&= (w \cdot\g)_\g(\xi) \frac{\dd\Phi}{\dd\xi} (\g(\xi),\g)+ \left(w\cdot\Phi\right)_\g (\g(\xi),\g)\\
&=  w_\g(\xi)\frac{\dd\Phi}{\dd\xi} (\g(\xi),\g)+  \left(w\cdot\Phi\right)_\g (\g(\xi),\g)
\end{align*}
where the composition of derivatives in the second equality can be checked by using the flow of $w$ in $\Diff$. The function $\Phi$ is thus a composition if and only if
\[
\left(w\cdot\Phi\right)_\g (\g(\xi),\g) = - w_\g(\xi)\frac{\dd\Phi}{\dd\xi} (\g(\xi),\g)
\]
that is to say, by composing in $\xi$ by $\g^{-1}$:
\[
\left(w\cdot\Phi\right)_\g (\xi,\g) = -w_\g\circ\g^{-1}(\xi)\frac{\dd \Phi}{\dd\xi}(\xi,\g).
\]
\end{proof}

Since for all $\g$, the vectors $\left(L_m\right)_\g$ defined by~\eqref{def-vir} generate the tangent space $T_\g\Diff$, we can deduce the following characterization of isomonodromic deformations.
\begin{cor}
A fundamental solution of the family of systems $\left(A_\g, \g\in\U\right)$ is $M$-invariant if and only if its monodromy satisfies for all $(\xi,\g)\in\S\times\U$
\[
L_m\cdot G (\xi,\g) = -\xi^{m+1}\frac{\dd G}{\dd\xi} (\xi,\g).
\]
\label{cor-def-isomono}
\end{cor}

From now on, the strategy to establish the universal Schlesinger system is quite similar to the one to establish the finite dimensional one: we exhibit a Pfaffian system which describes isomonodromic deformations of system~\eqref{Ag} and whose integrability condition is equivalent to the universal Schlesinger system.

\begin{prop} Let $Y(x,\g)$ be a fundamental solution such that $Y(\infty,\g)=\I_N$. For all $m\in\Z$, define $\Omega_m:\left(\P\ssm\S\right)\times\U\to\MM_N(\C)$ to be the function
\[
\Omega_m(x,\g) := \left(L_m\cdot Y(x,\g)\right)\cdot Y(x,\g)^{-1}.
\]
Then, $Y(x,\g)$ is $M$-invariant if and only if the functions $\Omega_m$ are given by
\begin{equation}
\Omega_m(x,\g) = \frac{-1}{2i\pi} \oint_\S \frac{\xi^{m+1}A(\xi,\g)}{\xi-x}\dd \xi.
\label{omega-cauchy}
\end{equation}
\end{prop}

In particular, $\Omega_{-1}(x,\g)=-\A(x,\g)$.

\begin{proof}
By definition, the function $\Omega_m(\cdot, \g)$ is well defined and analytic on the domains $\D$ and $\CD$, and verifies $\Omega_m(\infty, \g)=0$. Moreover it admits pointwise limits at any $\xi\in\S$ from the plus and minus side of $\S$, which satisfy
\[
\Omega_m^+(\xi, \g) =  \lim_{\substack{x\to\xi\\ x\in\mathbb D^+}} \left(L_m\cdot Y(x,\g)\right) \cdot Y^+(\xi,\g)^{-1} = \left(L_m\cdot Y^+(\xi,\g)\right)\cdot Y^+(\xi,\g)^{-1}.
\]
Thus on $\S\times\U$
\[
\Omega_m^+ = \left(L_m\cdot \left(Y^-G\right)\right) G^{-1} \left(Y^-\right)^{-1}= \Omega_m^- + Y^- L_m\left(G\right)G^{-1} \left(Y^-\right)^{-1}.
\]
Thanks to the Plemelj--Sokhotskii formula, we get
\[
\Omega_m(x,\g) = \frac{1}{2i\pi} \oint_\S \frac{\omega_m(\xi,\g)}{\xi-x}\dd \xi \, ,
\]
where $\omega_m=Y^- L_m\left(G\right)G^{-1} \left(Y^-\right)^{-1}$. But from~\eqref{A-jump} and from Corollary~\ref{cor-def-isomono}, we see that the fundamental solution $Y(x,\g)$ is $M$-invariant if and only if $\omega_m(\xi,\g)= -\xi^{m+1}A(\xi,\g)$.
\end{proof}

We can then deduce from~\eqref{omega-cauchy} the series expansions of the functions $\Omega_m(\cdot, \g)$.

\begin{cor}
For all $\g\in\U$ the expression~\eqref{omega-cauchy} of the functions $\Omega_m(\cdot, \g)$ is equivalent to 
\begin{align*}
\Omega_m(x,\g) &= -\sum_{n\geq1}\frac{B_{m+n}(\g)}{x^n} &\text{if } \vert x\vert >1,\\
\Omega_m(x,\g) &= \sum_{n\geq0} B_{m-n}(\g) x^n &\text{if } \vert x\vert <1.
\end{align*} 
\label{cor-omega-expansion}
\end{cor}

The family of differential systems $\left(A_\g, \g\in\U\right)$ is thus isomonodromic if and only if the following infinite dimensional Pfaffian system on $\left(\P\ssm\S\right)\times\U$
\begin{equation}
 \left\{
 \begin{aligned}
&\frac{\dd Y}{\dd x} = \A(x,\g) Y\\
&L_m\cdot Y = \Omega_m(x,\g) Y \qquad (m\in\Z)
 \end{aligned}
 \right.
\label{pfaff}
\end{equation}
where $\Omega_m(x,\g)$ is given by~\eqref{omega-cauchy}, is completely integrable. The integrability condition is given by the following Lemma.

\begin{lemma}
The Pfaffian system~\eqref{pfaff} is integrable on $\left(\P\ssm\S\right)\times\U$ if and only if:
\begin{align}
L_m\cdot \A -  \frac{\dd \Omega_m}{\dd x}  &= \left[ \Omega_m,\A\right] & (m\in\Z)\label{inte1}\\
L_m\cdot \Omega_n - L_n\cdot \Omega_m& =   \left[ \Omega_m,\Omega_n\right]  + (n-m)\Omega_{m+n} & (m,n\in\Z).\label{inte2}
\end{align} 
\end{lemma}

\begin{proof}
The commutation relations among the vector fields $\partial_x$ and $L_m$ are given by
\[
\left[\partial_x, L_m\right] = 0 , \qquad \left[L_m, L_n\right] = (n-m)L_{m+n} \qquad (m, n\in\Z).
\]
On integral submanifolds of the Pfaffian system~\eqref{pfaff}, we thus have
\begin{itemize}
\item $\left[\partial_x, L_m\right] Y= 0$ implies $
0 =\frac{\dd \Omega_m}{\dd x} - L_m\cdot \A + \left[ \Omega_m,\A\right] $,
\item and $\left[L_m, L_n\right] Y = (n-m)L_{m+n}\cdot Y$ gives
\[
L_m\cdot\left(\Omega_nY\right)-L_n\cdot\left(\Omega_mY\right) = (n-m)\Omega_{m+n}Y
\]
which is equivalent to~\eqref{inte2}.
\end{itemize}
Conversely, since the group $\Diff$ is a Banach manifold, we can apply Fröbenius theorem on $\left(\P\ssm\S\right)\times\U$, which asserts that the conditions~\eqref{inte1} and \eqref{inte2} guarantee the existence of a solution $Y(x,\g)$  to the Pfaffian system~\eqref{pfaff}  such that $Y(\infty,e)=\I_N$. Notice that $Y$ is not unique since $\left(\P\ssm\S\right)\times\U$ is not connected. Since $\Omega_m(\infty, \g)=0$, we then have $Y(\infty,\g)=\I_N$ for all $\g\in\U$.
\end{proof}

\begin{proof}[Proof of Theorem~\ref{thm-sch-univ}]
The last part that remains to be proven is that the integrability condition~\eqref{inte1}-\eqref{inte2} is equivalent to the universal Schlesinger system~\eqref{univ-sch}. This is obtained by identifying the coefficients of the series expansions at $0$ and $\infty$ of both sides of \eqref{inte1}-\eqref{inte2}. It thus relies on the series expansion~\eqref{A-expansion} of $\A(x,\g)$ and the one of the functions $\Omega_m(x,\g)$ given by Corollary~\ref{cor-omega-expansion}.

\bigskip

We first prove that the condition~\eqref{inte1} is actually equivalent to the universal Schlesinger system~\eqref{univ-sch}. For $x\in \CD$, $\g\in\U$, and for all $m\in\Z$ we have on one hand
\begin{align*}
\left[ \Omega_m(x,\g),\A(x,\g)\right] &= \left[ -\sum_{n\geq1}\frac{B_{m+n}(\g)}{x^{n}},  \sum_{n\geq0} \frac{B_n(\g)}{x^{n+1}}\right]\\
&= \sum_{n\geq1}  \frac{1}{x^{n+1}} \sum_{k=0}^{n-1} [B_k(\g),B_{m+n-k}(\g)]\, .
\end{align*}
On the other hand
\begin{align*}
\left(L_m\cdot \A\right)(x,\g) -  \frac{\dd \Omega_m}{\dd x}(x,\g) &= \sum_{n\geq0} \frac{L_m\cdot B_n(\g)}{x^{n+1}} -\sum_{n\geq1} \frac{n}{x^{n+1}} B_{m+n}(\g)\\
&=\sum_{n\geq0} \frac{1}{x^{n+1}}\left(L_m\cdot B_n(\g) -nB_{m+n}(\g) \right)\, .
\end{align*}
So by identifying the coefficients, we see that the integrability condition~\eqref{inte1} restricted onto $\CD\times\U$ is equivalent to $L_m\cdot B_0=0$ and 
\[
\forall n \geq1 \ \forall m\in\Z \quad L_m\cdot B_n = \sum_{k=0}^{n-1} [B_k,B_{m+n-k}] + nB_{m+n}
\]
on $\U$. The case $n\leq -1$ is given by the same way by the integrability condition~\eqref{inte1} on $\D\times\U$.

\bigskip

To end the proof of the theorem, it now suffices to show that System~\eqref{univ-sch} always implies the second integrability condition~\eqref{inte2}. We proceed as above: for $x\in \CD$, $\g\in\U$, and for all $m, n \in\Z$ we have on one hand
\begin{align*}
\left(L_m\cdot\Omega_n\right)(x,\g) - \left(L_n\cdot\Omega_m\right)(x,\g) = \sum_{k\geq1} \frac{1}{x^{k}} \left(-L_m\cdot B_{n+k}(\g) + L_n\cdot B_{m+k}(\g)\right).
\end{align*}
On the other hand
\begin{multline*}
\left[ \Omega_m(x,\g),\Omega_n(x,\g)\right]  + (n-m)\Omega_{m+n} (x,\g)\\
= \sum_{k,l\geq1}  \frac{1}{x^{k+l}}\left[ B_{m+k},B_{n+l}\right] - (n-m)\sum_{k\geq1} \frac{B_{m+n+k}}{x^k}\\
= \sum_{k\geq2}  \frac{1}{x^k} \left(  \sum_{l=1}^{k-1} [B_{m+k-l},B_{n+l}] \right) - (n-m)\sum_{k\geq1} \frac{B_{m+n+k}}{x^k} \, .
\end{multline*}
The integrability condition~\eqref{inte2} restricted onto $\CD\times\U$ is thus equivalent to 
\begin{equation}
L_m\cdot B_{n+k}- L_n\cdot B_{m+k} = -\sum_{l=1}^{k-1} [B_{m+k-l},B_{n+l}] + (n-m)B_{m+n+k} 
\label{inte2b}
\end{equation}
(for all $m, n \in\Z$ and $k\geq1$) where the sum is assumed to be $0$ when $k=1$. Assume now $n\geq1$. Then the equation number $(m,n+k)$ of the Schlesinger system~\eqref{univ-sch} writes
\[
L_m\cdot B_{n+k} = \sum_{l=0}^{n+k-1} \left[B_l,B_{m+n+k-l}\right] + (n+k)B_{m+n+k} \, ,
\]
whereas the equation number $(n,m+k)$ of the equivalent system~\eqref{univ-sch-2} writes
\[
L_n\cdot B_{m+k} = \sum_{l=0}^{n} \left[B_l,B_{m+n+k-l}\right] + (m+k)B_{m+n+k} \, .
\]
We thus have
\[
L_m\cdot B_{n+k}- L_n\cdot B_{m+k} =  \sum_{l=n+1}^{n+k-1} \left[B_l,B_{m+n+k-l}\right] + (n-m)B_{m+n+k} 
\]
that is to say the condition~\eqref{inte2b} for $m\in\Z$, $n\geq1$, $k\geq1$. By symmetry in $(m,n)$, we have \eqref{inte2b} for $n\in\Z$, $m\geq0$, $k\geq1$ too. Using the proof of lemma~\ref{lemma-sym}, we obtain by the same way the case $m\leq-1$ and $n\leq-1$.

This proves that the Schlesinger system~\eqref{univ-sch} implies the integrability condition~\eqref{inte2} on $\CD\times\U$. On $\D\times\U$, the proof is similar.
\end{proof}


\begin{thebibliography}{10}

\bibitem{DeiftZhou02}
P.~Deift and X.~Zhou.
\newblock A priori ${L}^p$-estimates for solutions of {R}iemann--{H}ilbert
  problems.
\newblock {\em International Mathematics Research Notices},
  2002(40):2121--2154, 2002.

\bibitem{Desideri13}
L.~Desideri.
\newblock Probl\`eme de {P}lateau, \'equations fuchsiennes et probl\`eme de
  {R}iemann-{H}ilbert.
\newblock {\em M\'em. Soc. Math. Fr. (N.S.)}, (133):vi+116, 2013.

\bibitem{FIKN}
A.~Fokas, A.~Its, A.~Kapaev, and V.~Novokshenov.
\newblock {\em Painlev{\'e} Transcendents: The Riemann--Hilbert Approach},
  volume 128.
\newblock American Mathematical Society, 2006.

\bibitem{Giorgadze}
G.~Giorgadze.
\newblock On the {R}iemann--{H}ilbert problems.
\newblock {\em arXiv: math/9804035}, 1998.

\bibitem{Guest}
M.~Guest.
\newblock {\em Harmonic maps, loop groups, and integrable systems}, volume~38.
\newblock Cambridge University Press, 1997.

\bibitem{KorotkinSamtleben}
D.~Korotkin and H.~Samtleben.
\newblock Generalization of {O}kamoto's equation to arbitrary {$2\times 2$}
  {S}chlesinger system.
\newblock {\em Adv. Math. Phys.}, pages Art. ID 461860, 14, 2009.

\bibitem{LitvinchukSpitkovskii}
G.~Litvinchuk and I.~Spitkovskii.
\newblock {\em Factorization of measurable matrix functions}, volume~25.
\newblock Birkhauser, 1987.

\bibitem{PressleySegal}
A~Pressley and G~Segal.
\newblock {\em Loop groups}.
\newblock Oxford Mathematical Monographs, Clarendon Press, Oxford, 1986.

\bibitem{Teo}
L.-P. Teo.
\newblock Conformal mappings and dispersionless {T}oda hierarchy.
\newblock {\em Communications in Mathematical Physics}, 292(2):391--415, 2009.

\bibitem{Teo2}
L.-P. Teo.
\newblock Conformal mappings and dispersionless {T}oda hierarchy {II}: general
  string equations.
\newblock {\em Communications in Mathematical Physics}, 297(2):447--474, 2010.

\bibitem{VajiacUhlenbeck}
M.~Vajiac and K.~Uhlenbeck.
\newblock Virasoro actions and harmonic maps (after {S}chwarz).
\newblock {\em J. Differential Geom.}, 80(2):327--341, 2008.

\bibitem{WiegmannZabrodin}
P.~Wiegmann and A.~Zabrodin.
\newblock Conformal maps and integrable hierarchies.
\newblock {\em Communications in Mathematical Physics}, 213(3):523--538, 2000.

\end{thebibliography}

\end{document}